\documentclass{article}
\usepackage{amsmath,amsthm,amscd,amssymb}
\usepackage[mathscr]{eucal}
\usepackage{amssymb}
\usepackage{latexsym}

\theoremstyle{definition}
\newtheorem{Def}{Definition}[section]
\newtheorem{Thm}[Def]{Theorem}

\numberwithin{equation}{section}

\title{On $p$-divisibility of Fourier coefficients of Hermitian modular forms}

\date{}
\begin{document}

\maketitle
\author{Shoyu Nagaoka}
\footnote{
Shoyu Nagaoka,\quad shoyu1122.sn@gmail.com\\
emeritus professor of Kindai University, 
Osaka 545-0001, Japan
}

\begin{abstract}
\noindent
We describe the $p$-divisibility transposition for the Fourier coefficients
of Hermitian modular forms. The results show that the same phenomenon
as that for Siegel modular forms
holds for Hermitian modular forms.
\end{abstract}



\section{Introduction}
\label{sec1}
This paper is a continuation of \cite{N}. In \cite{N}, we studied how the $p$-divisibility
of Fourier coefficients of Siegel modular forms shifts with changes in the
degree by
the Siegel operator $\Phi$.
As an application, provide a note on the $p$-divisibility of Ramanujan's $\tau$-function.

Let $F$ be a Siegel modular form with the Fourier expansion $F(Z)=\sum a(F,T)q^T$
whose Fourier coefficients $a(F,T)$ are $p$-integral. We observed the following
$p$-divisiblity transpositions:
\begin{align*}
& {\rm (I)}_p\qquad \text{All $a(F,T^{(n)})$ for $T^{(n)}>0$ are divisible by $p$.}\\
&\hspace{4cm} \Downarrow\; \Phi\\
& {\rm (II)}_p\qquad \text{``Most'' $a(\Phi(F),T^{(n-1)})$ for $T^{(n-1)}>0$ are divisible by $p$.}\\
&\hspace{4cm} \Downarrow\; \Phi\\
& {\rm (III)}_p\qquad \text{Certain ``special'' $a(\Phi^2(F),T^{(n-2)})$ for $T^{(n-2)}>0$ are divisible by $p$.}
\end{align*}
The following is an example of such a sequence of modular forms that behaves 
in this manner,
as described in \cite{N}:
$$
{\rm (I)}_{p}:\;\mathcal{E}^{(3)}(Z,\Delta)\overset{\Phi}{\longrightarrow}
{\rm (II)}_{p}:\;\mathcal{E}^{(2)}(Z,\Delta)
                                \overset{\Phi}{\longrightarrow} 
{\rm (III)}_{p}:\; \Delta\qquad (p=23),
$$
where $\Delta=\sum \tau(n)q^n$ is the cusp form of weight $12$, and the Fourier coefficient
$\tau (n)$ is the Ramanujan $\tau$-function (for precise definition of modular forms
$\mathcal{E}^{(n)}(Z,\Delta)$, see \cite{N}, Proposition 5).

One of the main results of \cite{N} is the following.
\vspace{1mm}
\\
\textbf{Result in the case of Siegel modular forms}\;(\cite{N}, Theorem M-3). 
{\it
Assume that $n$ is an even positive integer and $p$ is a prime number satisfying $p>n+3$
and $p \equiv (-1)^{n/2} \pmod{4}$. We set $k=(n+p-1)/2$. Then, there are Siegel modular
forms $F_k^{(n+1)}$ and $F_k^{(n)}$ of weight $k$ and respective degrees $n+1$ and $n$
satisfying $\Phi(F_k^{(n+1)})=F_k^{(n)}$, and
\vspace{1mm}
\\
{\rm S-(I)}${}_p{\rm :}\;\;a(F_k^{(n+1)},T^{(n+1)}) \equiv 0 \pmod{p}$ for all $T^{(n+1)}>0$,
\\
{\rm S-(II)}${}_p{\rm :}\;\;a(F_k^{(n)},T^{(n)}) \equiv 0 \pmod{p}$ for $T^{(n)}>0$ with 
${\rm det}(T^{(n)}) \not\equiv 0 \pmod{p}$.
}
\vspace{2mm}
\\
\quad In this paper, we report that similar results hold for the case of Hermitian modular forms.
Let $\boldsymbol{K}$ be an imagianry quadratic field. We consider a modular form $G(Z)$
for the Hermitian modular group $\Gamma_{\boldsymbol{K}}^{(m)}$ in $SU(m,m)$ with
Fourier expansion 
$$
G(Z)=\sum_{0\leq H\in \Lambda_m(\mathcal{O}_{\boldsymbol{K}})}a(G,H)q^H
$$
(for precise definition, see $\S$\,2). The main result of this paper is as follows.
\vspace{2mm}
\\
\textbf{Theorem.}\; {\it Assume that $m \equiv 2 \pmod{4}$ and $p$ is a prime
number satisfying $p>m+3$ and $D_{\boldsymbol{K}} \not\equiv 0 \pmod{p}$
($D_{\boldsymbol{K}}$ is the discriminant of $\boldsymbol{K}$). We set $k=m+p-1$.
Then, there are Hermitian modular forms $G_k^{(m+1)}$ and $G_k^{(m)}$ of weight $k$
and respective degrees $m+1$ and $m$ satisfying $\Phi(G_k^{(m+1)})=G_k^{(m)}$, and
\vspace{1mm}
\\
{\rm H-(I)}${}_p{\rm :}\;\;a(G_k^{(m+1)},H^{(m+1)}) \equiv 0 \pmod{p}$ for all $H^{(m+1)}>0$,
\\
{\rm H-(II)}${}_p{\rm :}\;a(G_k^{(m)},H^{(m)}) \equiv 0 \pmod{p}$ for $H^{(m)}>0$ with 
${\rm det}(H^{(m)}) \not\equiv 0 \pmod{p}$.
}
\vspace{1mm}
\\
\quad In the case $m=2$, we can construct examples of these forms by the Hermitian Eisenstein series.
\vspace{2mm}
\\
\textbf{Example.}\; Let $p>5$ be a prime number satisfying 
$$
h_{\boldsymbol{K}}\not\equiv 0 \pmod{p}\quad \text{and}\quad
\chi_{\boldsymbol{K}}(p)=-1,
$$
where $h_{\boldsymbol{K}}$ is the class number of $\boldsymbol{K}$ and
$\chi_{\boldsymbol{K}}$ is the quadratic Dirichlet character corresponding to
$\boldsymbol{K}/\mathbb{Q}$.
Let $E_{k,\boldsymbol{K}}^{(m)}$ denote the Hermitian Eisenstein series of weight $k$ and degree $m$
for $\boldsymbol{K}$ ( for definition, see $\S$\,\ref{sec2.2}). Then,
\vspace{1mm}
\\
{\rm Ex-H-(I)}${}_p{\rm :}\;\;a(E_{p+1,\boldsymbol{K}}^{(3)},H^{(3)}) \equiv 0 \pmod{p}$ for all $H^{(3)}>0$.
\\
{\rm Ex-H-(II)}${}_p{\rm :}$\\
\hspace{8mm} $a(E_{p+1,\boldsymbol{K}}^{(2)},H^{(2)}) \equiv 0 \pmod{p}$ for $H^{(2)}>0$ with 
${\rm det}(H^{(2)}) \not\equiv 0 \pmod{p}$,
\vspace{1mm}
\\
Additionally,
\vspace{1mm}
\\
{\rm Ex-H-(III)}${}_p{\rm :}$\\
\hspace{8mm} $a(E_{p+1,\boldsymbol{K}}^{(1)},h) \equiv 0 \pmod{p}$ for $h\in\mathbb{Z}_{>0}$ with 
$\sigma_1(h) \equiv 0 \pmod{p}$,
\vspace{1mm}
\\
where $\sigma_1(n)=\sum_{0<d|n}d$. (see $\S$\,\ref{sec3.3}).
\section{Hermitian modular forms}
Let $m$ be a positive integer and $\boldsymbol{K}=\mathbb{Q}(\sqrt{-D_{\boldsymbol{K}}})$
be an imaginary quadratic field with discriminant $-D_{\boldsymbol{K}}<0$.
We denote the ring of integers of $\boldsymbol{K}$ by $\mathcal{O}_{\boldsymbol{K}}$.
Let $\chi_{\boldsymbol{K}}$ be the quadratic Dirichlet character of conductor $D_{\boldsymbol{K}}$
corresponding to the extension $\boldsymbol{K}/\mathbb{Q}$. 
We denote the idele class character corresponding to
$\chi_{\boldsymbol{K}}$ by 
$\underline{\chi}_{\boldsymbol{K}}=\prod_v\underline{\chi}_{\boldsymbol{K},v}$.
\subsection{Hermitian modular forms and their Fourier expansion}
\label{sec2.1}
For a $\mathbb{Q}$-algebra $R$, the group $SU(m,m)(R)$ is given as
$$
SU(m,m)(R)
=\left\{g\in SL_{2m}(R\otimes_{\mathbb{Q}}\boldsymbol{K})\mid
g^*\left(\begin{smallmatrix} 0_m & -1_m \\ 1_m & 0_m \end{smallmatrix}\right) g=
\left(\begin{smallmatrix} 0_m & -1_m \\ 1_m & 0_m \end{smallmatrix}\right) \right\},
$$
where $g^*={}^t\overline{g}$. We set 
$$
\Gamma_{\boldsymbol{K}}^{(m)}=SU(m,m)(\mathbb{Q})\cap SL_{2m}(\mathcal{O}_{\boldsymbol{K}}).
$$
We denote the space of Hermitian modular forms
of weight $k$ for $\Gamma_{\boldsymbol{K}}^{(m)}$ by $M_k(\Gamma_{\boldsymbol{K}}^{(m)})$. 
Any modular form $F$ in 
$M_k(\Gamma_{\boldsymbol{K}}^{(m)})$ has a Fourier expansion of the form
$$
F(Z)=\sum_{0\leq H\in\Lambda_m(\mathcal{O}_{\boldsymbol{K}})}a(F,H)q^H,\;\;
q^H=\text{exp}(2\pi i\text{tr}(HZ)),\;\;Z\in\mathcal{H}_m,
$$
where
$$
\mathcal{H}_m=\{Z\in M_m(\mathbb{C})\mid \tfrac{1}{2i}(Z-Z^*)>0\}\;\;(\text{the Hermitian upper half space}),
$$
$$
\Lambda_m(\mathcal{O}_{\boldsymbol{K}})
=\{H=(h_{jl})\in M_m(\boldsymbol{K})\mid H^*=H,\,h_{jj}\in\mathbb{Z},\,\,\sqrt{-D_{\boldsymbol{K}}}\,h_{jl}\in
\mathcal{O}_{\boldsymbol{K}}\}.
$$
We also set $\Lambda_m^+(\mathcal{O}_{\boldsymbol{K}})=
\{H\in\Lambda_m(\mathcal{O}_{\boldsymbol{K}})\mid H>0\}$. For a subring $R$ of $\mathbb{C}$,
we set
$$
M_k(\Gamma_{\boldsymbol{K}}^{(m)})_R:=
\{F\in M_k(\Gamma_{\boldsymbol{K}}^{(m)})\mid a(F,H)\in R\;\; \text{for all}\;\; H\in \Lambda_m(\mathcal{O}_{\boldsymbol{K}})\}.
$$
As in the case of Siegel modular forms, we can define the Siegel operator
$\Phi :M_k(\Gamma_{\boldsymbol{K}}^{(m)}) \longrightarrow M_k(\Gamma_{\boldsymbol{K}}^{(m-1)})$.
For the Fourier coefficients, the following relation holds:
$$
a(\Phi(G),H^{(m-1)})=a\left(G,
     \left(\begin{smallmatrix} H^{(m-1)} & 0 \\ 0 & 0 \end{smallmatrix}\right) \right).
$$
\subsection{Hermitian Eisenstein series}
\label{sec2.2}
We set
$$
\Gamma_{\boldsymbol{K},\infty}^{(m)}=
\{\left(\begin{smallmatrix} A & B \\ C & D \end{smallmatrix}\right)\mid C=0_m\}.
$$
For a positive even integer $k>2m$, we define a Hermitian Eisenstein series of weight $k$ by
$$
E_{k,\boldsymbol{K}}^{(m)}(Z)
=\sum_{\left(\begin{smallmatrix} * & * \\ C & D \end{smallmatrix}\right)\in
\Gamma_{\boldsymbol{K},\infty}^{(m)}\backslash \Gamma_{\boldsymbol{K}}^{(m)}} 
\text{det}(CZ+D)^{-k},\quad Z\in\mathcal{H}_m.
$$
The Eisenstein series $E_{k,\boldsymbol{K}}^{(m)}(Z)$ is a typical example of
a Hermitian modular form; in fact, 
$E_{k,\boldsymbol{K}}^{(m)}(Z)\in M_k(\Gamma_{\boldsymbol{K}}^{(m)})_{\mathbb{Q}}$.
Moreover, $\Phi(E_{k,\boldsymbol{K}}^{(m)})=E_{k,\boldsymbol{K}}^{(m-1)}$ holds.
In the following, certain expression of the Fourier coefficient $a(E_{k,\boldsymbol{K}}^{(m)};H)$
is provided.

For a prime number $q$, we set 
$\mathcal{O}_{\boldsymbol{K},q}=\mathcal{O}_{\boldsymbol{K}}\otimes_{\mathbb{Z}}\mathbb{Z}_q$ and set
\begin{align*}
& \Lambda_n(\mathcal{O}_{\boldsymbol{K},q})\\
&=\{H=(h_{jl})\in M_m(\boldsymbol{K}\otimes_{\mathbb{Q}}\mathbb{Q}_q)\mid 
     H^*=H,\,h_{jj}\in\mathbb{Z}_q,\,\sqrt{-D_{\boldsymbol{K}}}h_{jl}\in\mathcal{O}_{\boldsymbol{K},q}\}
\end{align*}

Let $H\in\Lambda_m(\mathcal{O}_{\boldsymbol{K}})$ with $H\geq 0$ and set $r=\text{rank}_{\boldsymbol{K}}H$.
For each prime number $q$, we take $U_q\in GL_m(\mathcal{O}_{\boldsymbol{K},q})$, such that
$$
H[U_q]=\begin{pmatrix} H'_q & 0 \\ 0 & 0  \end{pmatrix}
$$
with $H'_q\in \Lambda_r(\mathcal{O}_{\boldsymbol{K},q})$. Here, for 
$A,B\in \text{Res}_{\boldsymbol{K}/\mathbb{Q}}M_m$, we define
$$
A[B]:=B^*AB.
$$
For $H\in \Lambda_r(\mathcal{O}_{\boldsymbol{K},q})$ with $\text{det}(H)\ne 0$, we denote
the polynomial given in Ikeda \cite{Ikeda}, $\S$\,2 by $\mathcal{F}_q(H,X)\in\mathbb{Z}[X]$.
Then, the polynomial $\mathcal{F}(H'_q,X)$ does not depend on the choice of $U_q$.
Therefore, we denote it by $\mathcal{F}_q(H,X)$. For $H\in \Lambda_r(\mathcal{\mathcal{O}}_{\boldsymbol{K}})$
(resp. $\in \Lambda_r(\mathcal{O}_{\boldsymbol{K,q}})$) with $\text{det}(H)\ne 0$, we define
$$
\gamma(H)=(-D_{\boldsymbol{K}})^{[r/2]}\text{det}(H)\in\mathbb{Z}\;\;(\text{resp.} \in\mathbb{Z}_q).
$$
The Fourier coefficient of the Hermitian Eisenstein series $E_{k,\boldsymbol{K}}^{(m)}$ has the
following description (for example, Ikeda \cite{Ikeda}).
\begin{Thm} 
\label{Ikeda1}
{\it
Let $H\in \Lambda_m(\mathcal{O}_{\boldsymbol{K}})$ with $H\geq 0$ and set
$r={\rm rank}_{\boldsymbol{K}}(H)$. Then, the Fourier coefficient $a(E_{k;\boldsymbol{K}}^{(m)},H)$
has the following expression:
$$
 a(E_{k,\;\boldsymbol{K}}^{(m)},H)
 =2^r\cdot \prod_{i=0}^{r-1}\left(-\frac{B_{k-i,\chi_{\boldsymbol{K}}^i}}{k-i}\right)^{-1}\cdot
       \prod_{q:{\rm prime}} \mathcal{F}_q(H,q^{k-2r} ).
$$
 Here $B_{*,\chi_{\boldsymbol{K}}^i}$ is the generalized Bernoulli number with  the character
 $\chi_{\boldsymbol{K}}^i$, and $B_{*,\chi_{\boldsymbol{K}}^i}=B_*$
 (the ordinary Bernoulli number) if $i$ is even.
}
\end{Thm}
\section{$p$-divisibility of Hermitian modular forms}
\label{sec3}
As stated in the Introduction, we can construct Hermitian modular forms that describe the transposition
of $p$-divisibility of Fourier coefficients. We denote
the local ring of $p$-integral rational numbers by 
$\mathbb{Z}_{(p)}=\mathbb{Q}\cap\mathbb{Z}_p$.
\begin{Thm} 
\label{main}
{\it
Assume that $m \equiv 2 \pmod{4}$ and $p$ is a prime number satisfying $p>m+3$
and $D_{\boldsymbol{K}} \not\equiv 0\pmod{p}$. We set $k=m+p-1$. Then, there are
modular forms $G_k^{(m+1)}\in M_k(\Gamma_{\boldsymbol{K}}^{(m+1)})_{\mathbb{Z}_{(p)}}$
and $G_k^{(m)}\in M_k(\Gamma_{\boldsymbol{K}}^{(m)})_{\mathbb{Z}_{(p)}}$
satisfying $\Phi(G_k^{(m+1)})=G_k^{(m)}$ and
\vspace{2mm}
\\
{\rm H-(I)}${}_p{\rm :}\;\;a(G_k^{(m+1)},H^{(m+1)}) \equiv 0 \pmod{p}$ for all 
$H^{(m+1)}\in \Lambda_{m+1}^+(\mathcal{O}_{\boldsymbol{K}})$.
\vspace{1mm}
\\
{\rm H-(II)}${}_p{\rm :}\;\;a(G_k^{(m)},H^{(m)}) \equiv 0 \pmod{p}$ for
$H^{(m)}\in \Lambda_{m}^+(\mathcal{O}_{\boldsymbol{K}})$ with ${\rm det}(H^{(m)})\not\equiv 0 \pmod{p}$.}
\end{Thm}
\begin{proof}
First of all, we construct $G_k^{(m)}\in M_k(\Gamma_{\boldsymbol{K}}^{(m)})_{\mathbb{Z}_{(p)}}$.
We set
$$
C_p:=
\text{ord}_p\left(\prod_{i=0}^{m-1}\left(-\frac{B_{k-i,\chi_{\boldsymbol{K}}^i}}{k-i}\right)^{-1}\right).
$$
From $D_{\boldsymbol{K}} \not\equiv 0\pmod{p}$, we see that $C_p\in\mathbb{Z}_{\leq 0}$
(see Carlitz \cite{Ca}, Theorem 1).
We set
$$
G_k^{(m)}:=p^{-C_p}\cdot E_{k,\boldsymbol{K}}^{(m)},
$$
where $E_{k,\boldsymbol{K}}^{(m)}$ is Hermitian Eisenstein series. From the choice of 
$C_p$,
we see that all Fourier coefficients of $G_k^{(m)}$ are $p$-integral. Specifically,
$G_k^{(m)}\in M_k(\Gamma_{\boldsymbol{K}}^{(m)})_{\mathbb{Z}_{(p)}}$.
We shall show that, if $H\in \Lambda_m^+(\mathcal{O}_{\boldsymbol{K}})$ satisfies 
$\text{det}(H)\not\equiv 0 \pmod{p}$, then
$$
a(G_k^{(m)},H) \equiv 0 \pmod{p}.
$$ 
The key result to show this is Ikeda's functional equation for $\mathcal{F}_q(H,X)$ (Ikeda \cite{Ikeda}).
With $m \equiv 2 \pmod{4}$, we see that $\gamma(H)<0$, and thus 
$\underline{\chi}_{\boldsymbol{K},\infty}(\gamma(H))=-1$. By the product formula for
the idele class character, there is a prime number $q$ such that 
$\underline{\chi}_{\boldsymbol{K},q}(\gamma(H))=-1$. Now, we apply Ikeda's functional
equation for $\mathcal{F}_q(H,X)$, which asserts that
$$
\mathcal{F}_q(H,q^{-2m}X^{-1})=\underline{\chi}_{\boldsymbol{K},q}(\gamma(H))^{m-1}
               (q^mX)^{-\text{ord}_q(\gamma(H))}\mathcal{F}_q(H,X)
$$
for $H\in\Lambda_m(\mathcal{O}_{\boldsymbol{K},q})$ with $\text{det}(H)\ne 0$. From
this equation, we obtain $\mathcal{F}_q(H,q^{-m})=0$. Since $\text{det}(H)\not\equiv 0\pmod{p}$,
we see that $\gamma(H) \not\equiv 0\pmod{p}$ (it should be noted that 
$D_{\boldsymbol{K}} \not\equiv 0 \pmod{p}$). This implies
\begin{align*}
 \prod_{q|\gamma(H)}\mathcal{F}_q(H,q^{k-2m})&=\prod_{q|\gamma(H)}\mathcal{F}_q(H,q^{p-m-1})\\
   & \equiv \prod_{q|\gamma(H)}\mathcal{F}_q(H,q^{-m})=0 \pmod{p}.                         
\end{align*}
Since
$$
p^{-C_p}\cdot \prod_{i=0}^{m-1}\left(-\frac{B_{k-i,\chi_{\boldsymbol{K}}}^i}{k-i}\right)^{-1}
\in \mathbb{Q}\cap \mathbb{Z}_p^\times,
$$
we obtain
\begin{align*}
a(G_k^{(m)},H) &= (\text{$p$-adic unit})\cdot  \prod_{q|\gamma(H)}\mathcal{F}_q(H,q^{k-2m})\\
              & \equiv 0 \pmod{p}.
\end{align*}

Secondly, we construct $G_k^{(m+1)}$. We set
$$
G_k^{(m+1)}:=p^{-C_p}\cdot E_{k,\boldsymbol{K}}^{(m+1)}
$$
where $C_p$ is the constant used in the definition of $G_k^{(m)}$. The fact
$\Phi(E_{k,\boldsymbol{K}}^{(m+1)})=E_{k,\boldsymbol{K}}^{(m)}$ implies
$\Phi(G_k^{(m+1)})=G_k^{(m)}$. The Fourier coefficient $a(G_k^{(m+1)},H)$ for 
$H\in\Lambda_{m+1}^+(\mathcal{O}_{\boldsymbol{K}})$ is expressed as
$$
2^{m+1}\cdot p^{-C_p}\cdot\prod_{i=0}^{m}\left(-\frac{B_{k-i,\chi_{\boldsymbol{K}}^i}}{k-i}\right)^{-1}\cdot
       \prod_{q:{\rm prime}} \mathcal{F}_q(H;q^{k-2(m+1)} ).
$$
Since the last product $\prod\mathcal{F}_q(\cdots)$ is in $\mathbb{Z}$ (note that $p>m+3$), 
it suffices to show that the rational number
$$
D:=p^{-C_p}\cdot \prod_{i=0}^{m}\left(-\frac{B_{k-i,\chi_{\boldsymbol{K}}^i}}{k-i}\right)^{-1}
$$
is divisible by $p$. We note that
$$
D=p^{-C_p}\cdot\prod_{i=0}^{m-1}\left(-\frac{B_{k-i,\chi_{\boldsymbol{K}}^i}}{k-i}\right)^{-1}
\cdot
\left(-\frac{B_{p-1}}{p-1}\right)^{-1}.
$$
By the theorem von Staudt and Clausen and the definition of $C_p$, we obtain
$$
\text{ord}_p(D)=0-\text{ord}_p\left(\frac{B_{p-1}}{p-1}\right)=1.
$$
This shows that
$$
a(G_k^{(m+1)},H) \equiv 0 \pmod{p}
$$
for all $H\in\Lambda_{m+1}^+(\mathcal{O}_{\boldsymbol{K}})$. This completes the proof.
\end{proof}
\section{Remark}
\label{sec3}
\subsection{Theta operators and mod $p$ singular forms}
\label{sec3.1}
The theta operator $\varTheta$ on the space of Hermitian modular forms is defined as
$$
\varTheta :\;F=\sum a(F,H)q^H\longrightarrow \varTheta(F):=\sum a(F,H)\cdot\text{det}(H)q^H,
$$
which has its origins in Ramanujan's theta operator in the case of elliptic modular forms.
The content of {\rm H-(II)}${}_p$ in Theorem \ref{main} can be paraphrased as
$$
\varTheta (G_k^{(m)}) \equiv 0 \pmod{p}.
$$
In this case, we say that $G_k^{(m)}$ is an element of {\it the mod $p$ kernel} of the theta opeartor.

If $F\in M_k(\Gamma_{\boldsymbol{K}}^{(m)})_{\mathbb{Z}_{(p)}}$ satisfies that
$a(F,H) \equiv 0 \pmod{p}$ for all $H\in \Lambda_m^+(\mathcal{O}_{\boldsymbol{K}})$,
then we say that $F$ is {\it a mod $p$ singular form}. The content of {\rm H-(I)}${}_p$
in Theorem \ref{main}
is rephrased as $G_k^{(m+1)}$ is a mod $p$ singular form. Additionally, if $F\in M_k(\Gamma_{\boldsymbol{K}}^{(m)})_{\mathbb{Z}_{(p)}}$ is a mod $p$ singular form, then $F$ is an
element of the mod $p$ kernel of theta operator. We call {\it essential} if $F$ a modular form
$F$ is included in the mod $p$ kernel, but it is not mod $p$ singular. It is known that $G_k^{(m)}$
is an essential form. In fact, $a(G_k^{(m)},H) \not\equiv 0 \pmod{p}$ for $H\in \Lambda_m^+(\mathcal{O}_{\boldsymbol{K}})$ with $\gamma(H)=-p$.

From Theorem \ref{main}, the following is expected: If $G\in M_k(\Gamma_{\boldsymbol{K}}^{(m)})_{\mathbb{Z}_{(p)}}$
is an essential form, then there exists some mod $p$ singular form 
$G^{\uparrow}\in M_k(\Gamma_{\boldsymbol{K}}^{(m+1)})_{\mathbb{Z}_{(p)}}$ satisfyng $\Phi(G^{\uparrow})=G$.
\subsection{Theta series for Hermitian lattices}
\label{sec3.2}
We can construct Hermitian modular forms in the mod $p$ kernel of the theta operator using
the theta series.

For a positive Hermitian lattice $\mathcal{L}$ of rank $r$, we associate the Hermitian theta series
$$
\vartheta_{\mathcal{L}}^{(m)}(Z)
=\vartheta_{H}^{(m)}(Z)
=\sum_{X\in \mathcal{O}^{(r,m)}}\text{exp}(\pi i\text{tr}(H[X]Z)),\;Z\in\mathcal{H}_m,
$$
where $H$ is the corresponding Gram matrix of $\mathcal{L}$.

In the following, we assume that $\boldsymbol{K}=\mathbb{Q}(i)$. We denote
the set of
even integral, positive definite unimodular Hermitian matrices of rank $r$
by $\mathcal{U}_r(\mathcal{O}_{\boldsymbol{K}})=\mathcal{U}_r(\mathbb{Z}[i])$. 
It is known that
$4|r$ and that there is a unimodular Hermitian lattice 
$\mathcal{L}_{\mathbb{C}}\in \mathcal{U}_{12}(\mathbb{Z}[i])$, which does not have any
vector of length one. The transfer of this lattice to $\mathbb{Z}$ is the Leech lattice
(cf. http://www.math.uni-sb.de/ag/schulze/Hermitian-lattices/). For this lattice, we have
$$
\left.\vartheta_{\mathcal{L}_{\mathbb{C}}}^{(2)}\right|_{\mathbb{S}_2}=\vartheta_{\text{Leech}}^{(2)},
$$
that is, the resriction of Hermitian theta series $\vartheta_{\mathcal{L}_{\mathbb{C}}}^{(2)}$ to the
Siegel upper half space $\mathbb{S}_2$ coincides with the Siegel theta series $\vartheta_{\text{Leech}}^{(2)}$
attached to the Leech lattice. We fix the lattice $\mathcal{L}_{\mathbb{C}}$ and call it the Hermitian
Leech lattice. It is known that 
$\vartheta_{\mathcal{L}_{\mathbb{C}}}^{(2)}\in M_{12}(\Gamma_{\mathbb{Q}(i)}^{(2)})_{\mathbb{Z}}$ and
$$
\Theta (\vartheta_{\mathcal{L}_{\mathbb{C}}}^{(2)}) \equiv 0 \pmod{11}
$$
(cf. \cite{K-N}, Theorem 8). It should be noted that the above defined $G_{12}^{(2)}$ also satisfies
$\Theta (G_{12}^{(2)}) \equiv 0 \pmod{11}$ ($m=2$, $p=11$). From the prediction in the previous
section, it is expected that $\vartheta_{\mathcal{L}_{\mathbb{C}}}^{(3)}$ is a mod $11$ singular
form, that is
$$
a(\vartheta_{\mathcal{L}_{\mathbb{C}}}^{(3)},H) \equiv 0 \pmod{11}
$$
for all $H\in \Lambda_3^+(\mathbb{Z}[i])$.
\subsection{Example}
\label{sec3.3}
In this section, we provide an example of the main theorem when $m=2$.

As mentioned in the Introduction, we fix a prime number $p>5$ satisfying 
$$
\chi_{\boldsymbol{K}}(p)=-1\quad \text{and}\quad h_{\boldsymbol{K}}\not\equiv 0
\pmod{p}.
$$
In this case, $k=m+p-1=p+1$, and the Hermitian Eisenstein series
$E_{p+1,\boldsymbol{K}}^{(3)}$ and $E_{p+1,\boldsymbol{K}}^{(2)}$ can be taken
as $G_{k,\boldsymbol{K}}^{(m+1)}$ and $G_{k,\boldsymbol{K}}^{(m)}$ of the main theorem.
That is, we obtain
\vspace{2mm}
\\
{\rm Ex-H-(I)}${}_p{\rm :}\;\;a(E_{p+1,\boldsymbol{K}}^{(3)},H^{(3)}) \equiv 0 \pmod{p}$ for all $H^{(3)}>0$,
\\
{\rm Ex-H-(II)}${}_p{\rm :}$\\
\hspace{8mm} $a(E_{p+1,\boldsymbol{K}}^{(2)},H^{(2)}) \equiv 0 \pmod{p}$ for $H^{(2)}>0$ with 
${\rm det}(H^{(2)}) \not\equiv 0 \pmod{p}$.
\vspace{2mm}
\\
\quad In this case, the constant $C_p$ appearing in the proof of the main theorem
is given by
$$
C_p=\text{ord}_p\left(\left(-\frac{p+1}{B_{p+1}}\right)\cdot
\left(-\frac{p}{B_{p,\chi_{\boldsymbol{K}}}} \right)  \right).
$$
From the way of proof, it is sufficient to show that this value is zero.

By Kummer's congruence relation, we obtain
\vspace{1mm}
\\
$\bullet\displaystyle \quad \frac{B_{p+1}}{p+1} \equiv \frac{B_2}{2}=\frac{1}{12} \pmod{p}$,
\vspace{1mm}
\\
$\bullet\displaystyle \quad \frac{B_{p,\chi_{\boldsymbol{K}}}}{p} 
\equiv (1-\chi_{\boldsymbol{K}}(p))B_{1,\chi_{\boldsymbol{K}}}
       =(1-\chi_{\boldsymbol{K}}(p))\frac{-2h_{\boldsymbol{K}}}{w_{\boldsymbol{K}}}
       \pmod{p}$,
\vspace{2mm}
\\
where $w_{\boldsymbol{K}}$ is the order of the unit group of $\boldsymbol{K}$.
By the assumption on $p$, both of the above two values are $p$-adic units.
Therefore, we obtain $C_p=0$.

Additionally, we obtain
\vspace{1mm}
\\
{\rm Ex-H-(III)}${}_p{\rm :}$\\
\hspace{8mm} $a(E_{p+1,\boldsymbol{K}}^{(1)},h) \equiv 0 \pmod{p}$ for $h\in\mathbb{Z}_{>0}$ with 
$\sigma_1(h) \equiv 0 \pmod{p}$,
\vspace{1mm}
\\
where $\sigma_1(n)$ is the sum of divisors of $n$, that is, $\sigma_1(n)=\sum_{0<d|n}d$. 

\end{document}